\documentclass[a4paper,twoside]{amsart}

\usepackage{verbatim}
\usepackage[british]{babel}
\usepackage[hyperfootnotes=false, plainpages=false, pdfpagelabels, colorlinks, linkcolor=red, citecolor=blue, urlcolor=blue, hypertexnames=true]{hyperref} 
\usepackage{amsmath}
\usepackage{amssymb}
\usepackage{amsthm}
\usepackage{amsfonts}
\usepackage{amsxtra}
\usepackage{url}
\usepackage{latexsym}
\usepackage{enumerate}

\renewcommand\epsilon\varepsilon
\renewcommand\phi\varphi
\newcommand\cb{\mathrm{cb}}
\newcommand\diag{\operatorname{diag}}

\newcommand\bbR{\mathbb{R}}

\newcommand\cs{C^{\ast}}

\DeclareMathOperator\SL{SL}

\DeclareMathOperator\Sp{Sp}

\newcommand\widecheck\check

\theoremstyle{definition}

\makeatletter
\newtheorem*{rep@theorem}{\rep@title}
\newcommand{\newreptheorem}[2]{%
\newenvironment{rep#1}[1]{%
 \def\rep@title{#2 \ref{##1}}%
 \begin{rep@theorem}}%
 {\end{rep@theorem}}}
\makeatother

\newtheorem{thmA}{Theorem}

\newtheorem{thm}{Theorem}[section]
\newreptheorem{thm}{Theorem}
\newtheorem{dfn}[thm]{Definition}
\newtheorem{lem}[thm]{Lemma}
\newtheorem{prp}[thm]{Proposition}
\newtheorem{cor}[thm]{Corollary}

\newtheorem{rmk}[thm]{Remark}

\author{Uffe Haagerup}
\thanks{Uf{}fe Haagerup tragically passed away on July 5, 2015.\\
{}\\
\indent UH was supported by ERC Advanced Grant no.~OAFPG 247321, the Danish Natural Science Research Council, the Danish National Research Foundation through the Centre for Symmetry and Deformation (DNRF92), and Villum Fonden Research Grant ``Local and Global Structures of Groups and their Algebras'' (2014--2018).}
\address{Uffe Haagerup
\newline University of Southern Denmark, Department of Mathematics and Computer Science
\newline Campusvej 55, 5230 Odense M, Denmark}
\email{}

\author{S\o ren Knudby}
\thanks{SK was supported by ERC Advanced Grant no.~OAFPG 247321 and the Danish National Research Foundation through the Centre for Symmetry and Deformation (DNRF92), and is currently supported by the
Deutsche Forschungsgemeinschaft through the Collaborative Research Centre (SFB 878)}
\address{S\o ren Knudby
\newline Mathematical Institute, University of M\"unster,
\newline Einsteinstrasse  62, 48149 M\"unster, Germany}
\email{knudby@uni-muenster.de}

\author{Tim de Laat}
\thanks{TdL is a Postdoctoral Fellow of the Research Foundation -- Flanders (FWO)}
\address{Tim de Laat
\newline KU Leuven, Department of Mathematics
\newline Celestijnenlaan 200B -- Box 2400, B-3001 Leuven, Belgium}
\email{tim.delaat@wis.kuleuven.be}

\title[Connected Lie groups with the Approximation Property]{A complete characterization of connected Lie groups with the Approximation Property}
\date{}

\begin{document}

\begin{abstract}
We give a complete characterization of connected Lie groups with the Approximation Property for groups (AP). To this end, we introduce a strengthening of property (T), that we call property (T$^{\ast}$), which is a natural obstruction to the AP. In order to define property (T$^{\ast}$), we first prove that for every locally compact group $G$, there exists a unique left invariant mean $m$ on the space $M_0A(G)$ of completely bounded Fourier multipliers of $G$. A locally compact group $G$ is said to have property (T$^{\ast}$) if this mean $m$ is a weak$^{\ast}$ continuous functional. After proving that the groups $\mathrm{SL}(3,\mathbb{R})$, $\mathrm{Sp}(2,\mathbb{R})$, and $\widetilde{\mathrm{Sp}}(2,\mathbb{R})$ have property (T$^{\ast}$), we address the question which connected Lie groups have the AP. A technical problem that arises when considering this question from the point of view of the AP is that the semisimple part of the global Levi decomposition of a connected Lie group need not be closed. Because of an important permanence property of property (T$^{\ast}$), this problem vanishes. It follows that a connected Lie group has the AP if and only if all simple factors in the semisimple part of its Levi decomposition have real rank $0$ or $1$. Finally, we are able to establish property (T$^{\ast}$) for all connected simple higher rank Lie groups with finite center.
\end{abstract}

\maketitle

\section{Introduction} \label{sec:introduction}
The main aim of this article is to provide a complete characterization of connected Lie groups with the Approximation Property for groups (AP). It continues and relies on the work of Lafforgue and de la Salle \cite{lafforguedelasalle} and the work of the first named and the third named author \cite{haagerupdelaat1}, \cite{haagerupdelaat2}.

A locally compact group $G$ has the AP if there is a net $(\varphi_{\alpha})$ in the Fourier algebra $A(G)$ of $G$ such that $\varphi_{\alpha} \to 1$ in the weak* topology on $M_0A(G)$ (see Section \ref{sec:preliminaries} for details). The AP, which was introduced by the first named author and Kraus in \cite{haagerupkraus}, is the natural analogue for groups of the Banach Space Approximation Property (BSAP) of Grothendieck (see \cite{grothendieck}). To see this, recall first that Banach spaces have a natural noncommutative analogue, namely, operator spaces. An operator space is a closed linear subspace $E$ of the bounded operators $\mathcal{B}(\mathcal{H})$ on a Hilbert space $\mathcal{H}$ (see \cite{effrosruanoperatorspaces}, \cite{pisieroperatorspaces}). For the class of operator spaces, which contains the class of $\cs$-algebras, a well-known analogue of the BSAP is known, namely, the operator space approximation property (OAP). The first named author and Kraus proved that a discrete group $\Gamma$ has the AP if and only if its reduced $\cs$-algebra $C^{\ast}_{\lambda}(\Gamma)$ has the OAP.

The AP is strictly weaker than the well-known properties amenability and weak amenability. It has not been considered as much as these properties, or the Haagerup property for that matter, probably because until recently, the only examples of groups without the AP followed from the theoretical fact that every discrete group with the AP is exact, as established by the first named author and Kraus. However, in 2010, Lafforgue and de la Salle provided the first concrete examples of groups without the AP, namely, $\mathrm{SL}(n,\mathbb{R})$ for $n \geq 3$ and lattices in these groups \cite{lafforguedelasalle}. In fact, they also proved that $\mathrm{SL}(n,F)$ (with $n \geq 3$) and its lattices do not have the AP for any non-Archimedean local field $F$. Their result on real Lie groups was extended by the first named and the third named author of this article, first to connected simple Lie groups with real rank at least $2$ and finite center \cite{haagerupdelaat1}, by proving that $\mathrm{Sp}(2,\mathbb{R})$ does not satisfy the AP, and then to all connected simple Lie groups with real rank at least $2$ \cite{haagerupdelaat2}, by also considering the universal covering group $\widetilde{\mathrm{Sp}}(2,\mathbb{R})$. Indeed, any connected simple Lie group with real rank at least $2$ contains a closed subgroup locally isomorphic to $\mathrm{SL}(3,\mathbb{R})$ or $\mathrm{Sp}(2,\mathbb{R})$, which, together with the known permanence properties of the AP, implies the general result. From this, it follows that a connected semisimple Lie group $G$ that is isomorphic to a direct product $S_1 \times \cdots \times S_n$ of connected simple Lie groups has the AP if and only if all these $S_i$'s have real rank $0$ or $1$ (see \cite{haagerupdelaat2}).

It is straightforward to characterize the AP for more general classes of groups, by using the known permanence properties of the AP, but we are not able to provide a complete characterization of connected Lie groups with the AP in this way. However, by using a natural obstruction to the AP that we introduce in this article, it turns out that we can. This obstruction, which we call property (T$^{\ast}$), is in fact a strengthening of property (T). Property (T), as introduced by Kazhdan in \cite{kazhdan}, is a rigidity property for groups that has lead to striking results in several areas of mathematics (see \cite{bekkadelaharpevalette}). Recall that a locally compact group has property (T) if its trivial representation is isolated in the unitary dual of the group equipped with the Fell topology.

As mentioned above, property (T$^{\ast}$) forms a natural obstruction to the AP, in the sense that a locally compact group having both the AP and property (T$^{\ast}$) is necessarily compact. Note that in the same way, property (T) is an obstruction to the Haagerup property.

For a locally compact group $G$, let $M_0A(G)$ denote the space of completely bounded Fourier multipliers of $G$ (see Section \ref{sec:preliminaries}). In order to define property (T$^{\ast}$), we first prove the following result. 
\begin{thmA} \label{thm:uniqueinvariantmeanm0a}
Let $G$ be a locally compact group. Then the space $M_0A(G)$ carries a unique left invariant mean $m$. This mean is also right invariant.
\end{thmA}
It is known that $M_0A(G)$ is a subspace of the space $W(G)$ of weakly almost periodic functions on $G$, which is known to have a unique left invariant mean. It follows that the mean on $M_0A(G)$ is the restriction to $M_0A(G)$ of the mean on $W(G)$. It is known from \cite[Chapitre III]{godement} that also the Fourier-Stieltjes algebra $B(G)$ (see Section \ref{sec:preliminaries}) of a locally compact group $G$ has a unique left invariant mean.

As mentioned before, $M_0A(G)$ carries a weak$^\ast$ topology (see Section \ref{subsec:ap}).
\begin{dfn}
A locally compact group $G$ is said to have property (T$^{\ast}$) if the unique left invariant mean $m$ on $M_0A(G)$ is a weak$^{\ast}$ continuous functional.
\end{dfn}
It is easy to see that compact groups have property (T$^{\ast}$). Also, as mentioned before, any group satisfying both the AP and property (T$^{\ast}$) is compact. Using a powerful result of Veech from \cite{veech} and the results of \cite{haagerupdelaat1} and \cite{haagerupdelaat2}, we are able to prove the following result.
\begin{thmA} \label{thm:sl3sp2sp2cov}
The groups $\mathrm{SL}(3,\mathbb{R})$, $\mathrm{Sp}(2,\mathbb{R})$, and the universal covering group $\widetilde{\mathrm{Sp}}(2,\mathbb{R})$ of $\mathrm{Sp}(2,\mathbb{R})$ have property (T$^{\ast}$).
\end{thmA}
Property (T$^{\ast}$) satisfies certain permanence properties. One of the essential ones for us is that whenever $\pi \colon H \rightarrow G$ is a continuous homomorphism between locally compact groups with dense image and $H$ has property (T$^{\ast}$), then $G$ has property (T$^{\ast}$) (see Proposition \ref{prp:denseimage}). Using Theorem \ref{thm:sl3sp2sp2cov} and this permanence property, we are able to prove the following theorem, which gives a complete characterization of connected Lie groups with the AP. The statement of the theorem uses the Levi decomposition of connected Lie groups (see Section \ref{subsec:levidecomposition}), asserting that any connected Lie group $G$ admits a decomposition $G=RS$, where $R$ is a solvable closed normal subgroup of $G$ and $S$ is a semisimple subgroup of $G$. The semisimple Lie group $S$ is locally isomorphic to a direct product of connected simple factors.
\begin{thmA} \label{thm:apconnectedliegroups}
Let $G$ be a connected Lie group, let $G=RS$ be a Levi decomposition, and suppose that $S$ is locally isomorphic to the direct product $S_1 \times \cdots \times S_n$ of connected simple factors. Then the following are equivalent:
\begin{enumerate}[(i)]
 \item the group $G$ has the AP,
 \item the group $S$ has the AP,
 \item the groups $S_i$, where $i=1,\ldots,n$, have the AP,
 \item the real rank of the groups $S_i$, where $i=1,\ldots,n$, is at most $1$.
\end{enumerate}
\end{thmA}
For the Haagerup property, an analogous characterization was given in \cite[Chapter 4]{cherixcowlingjolissaintjulgvalette}. For weak amenability, a characterization was given in \cite{cowlingdorofaeffseegerwright} under the assumption that the semisimple part of the Levi decomposition has finite center.

For the proof of Theorem \ref{thm:apconnectedliegroups}, we use the fact that a locally compact group with a non-compact closed subgroup with property (T$^{\ast}$) does not have the AP. Let us point out that it follows from the proof of Theorem~\ref{thm:apconnectedliegroups} (see Section~\ref{sec:connectedliegroups}) that a connected Lie group has the AP if and only if it contains no non-compact closed subgroups with property (T$^{\ast}$).

It turns out that we can actually generalize property (T$^{\ast})$ to connected simple higher rank Lie groups with finite center.
\begin{thmA} \label{thm:tstarhigherrank}
Let $G$ be a connected simple Lie group with real rank at least $2$ and finite center. Then $G$ has property (T$^{\ast}$).
\end{thmA}
We expect that this theorem is also true without the finite center condition.\\

The article is organized as follows. In Section \ref{sec:preliminaries}, we recall some preliminaries on the Fourier-Stieltjes algebra, the Fourier algebra, completely bounded Fourier multipliers, the AP, and the Levi decomposition. In Section \ref{sec:invariantmeans}, we recall some results on invariant means associated with locally compact groups and we prove Theorem \ref{thm:uniqueinvariantmeanm0a}. We discuss a characterization of property (T) in Section \ref{sec:propertyt}. We introduce property (T$^{\ast}$) in Section \ref{sec:propertytstar}, and we prove Theorem \ref{thm:sl3sp2sp2cov} in Section \ref{sec:sl3sp2sp2cov}. In Section \ref{sec:connectedliegroups}, we prove Theorem \ref{thm:apconnectedliegroups}. Finally, we prove Theorem \ref{thm:tstarhigherrank} in Section \ref{sec:tstarhigherrankliegroups}.

\section{Preliminaries} \label{sec:preliminaries}
\subsection{The Fourier-Stieltjes algebra and the Fourier algebra} \label{subsec:fsf}
Recall that the Fourier-Stieltjes algebra $B(G)$ of a locally compact group $G$ is the space of matrix coefficients of strongly continuous unitary representations of $G$. More precisely, $\varphi \in B(G)$ if there exist a strongly continuous unitary representation $\pi\colon G \rightarrow \mathcal{B}(\mathcal{H})$ and vectors $\xi,\eta \in \mathcal{H}$ such that for all $g \in G$,
\begin{equation} \label{eq:fscoefficients}
\varphi(g)=\langle \pi(g)\xi,\eta \rangle.
\end{equation}
With the norm $\|\varphi\|_{B(G)}=\inf\{\|\xi\|\|\eta\|\}$, where the infimum is taken over all vectors $\xi,\eta \in \mathcal{H}$ such that \eqref{eq:fscoefficients} holds, it forms a Banach space. Alternatively, the Fourier-Stieltjes algebra can be defined as the complex linear span of the continuous positive definite functions on $G$.

Let $\lambda:G \rightarrow \mathcal{B}(L^2(G))$ denote the left regular representation, which is given by $(\lambda(g)\xi)(h)=\xi(g^{-1}h)$, where $g,h \in G$ and $\xi \in L^2(G)$. Recall that the Fourier algebra $A(G)$ consists of the coefficients of $\lambda$. More precisely, $\varphi \in A(G)$ if there exist $\xi,\eta \in L^2(G)$ such that for all $g \in G$,
\begin{equation} \label{eq:fcoefficients}
	\varphi(g)=\langle \lambda(g)\xi,\eta \rangle.
\end{equation}
The Fourier algebra $A(G)$ is a Banach space when equipped with the norm $\|\varphi\|_{A(G)}=\inf\{\|\xi\|\|\eta\|\}$, where the infimum is taken over all vectors $\xi,\eta \in \mathcal{H}$ such that \eqref{eq:fcoefficients} holds. We have $\|\varphi\|_{A(G)} \geq \|\varphi\|_{\infty}$ for all $\varphi \in A(G)$, and $A(G)$ is $\|.\|_{\infty}$-dense in $C_0(G)$. It is well-known that the Fourier algebra $A(G)$ can be identified with the predual of the group von Neumann algebra $L(G)$ of $G$.

Both the Fourier-Stieltjes algebra and the Fourier algebra were introduced by Eymard \cite{eymard} (see also \cite{eymard2}).

\subsection{Fourier multipliers} \label{subsec:fm}
A function $\varphi \colon G \rightarrow \mathbb{C}$ is said to be a completely bounded (Fourier) multiplier if the map $M_{\varphi} \colon L(G) \rightarrow L(G)$ defined through $\lambda(g) \mapsto \varphi(g)\lambda(g)$ is completely bounded. The space of completely bounded multipliers is denoted by $M_0A(G)$, and with the norm $\|\varphi\|_{M_0A(G)}=\|M_{\varphi}\|_{\cb}$, where $\|.\|_{\cb}$ denotes the completely bounded norm, it is a Banach space. It is known that $A(G) \subset M_0A(G)$.

In \cite{bozejkofendler1}, Bo\.zejko and Fendler proved the equivalence between completely bounded Fourier multipliers and Herz-Schur multipliers, which were first studied by Herz \cite{herz}. They also gave another important characterization of completely bounded Fourier multipliers, namely, $\varphi \in M_0A(G)$ if and only if there exist bounded continuous maps $P,Q:G \rightarrow \mathcal{H}$, where $\mathcal{H}$ is a Hilbert space, such that
\begin{equation} \label{eq:cbmcharacterization}
  \varphi(h^{-1}g)=\langle P(g),Q(h) \rangle
\end{equation}
for all $g,h \in G$. In this characterization, $\|\varphi\|_{M_0A(G)}=\inf\{\|P\|_{\infty}\|Q\|_{\infty}\}$, where the infimum is taken over all possible pairs $(P,Q)$ such that \eqref{eq:cbmcharacterization} holds.

\subsection{The Approximation Property} \label{subsec:ap}
Let $G$ be a locally compact group with left Haar measure $dg$. For $f\in L^1(G)$ and $\varphi \in L^{\infty}(G)$, we set $\langle f,\varphi \rangle = \int f(g)\varphi(g)dg$. Let $X$ denote the completion of $L^1(G)$ with respect to the norm
\[
  \|f\| = \sup\left\{ |\langle f,\varphi \rangle | \mid \varphi \in M_0A(G), \|\varphi\|_{M_0A(G)} \leq 1 \right\}.
\]
Then $X^{\ast} \cong M_0A(G)$. Thus, $X$ is a predual of $M_0A(G)$, denoted by $M_0A(G)_{\ast}$. This predual defines the weak* topology on $M_0A(G)$. For details, we refer to \cite[Proposition 1.10]{decannierehaagerup}.
\begin{dfn}(\cite{haagerupkraus})
A locally compact group $G$ is said to have the Approximation Property for groups (AP) if there exists a net $(\varphi_{\alpha})$ in $A(G)$ such that $\varphi_{\alpha} \to 1$ in the weak* topology on $M_0A(G)$.
\end{dfn}
It was proved in \cite{haagerupkraus} that if $G$ is a locally compact group and $\Gamma$ is a lattice in $G$, then $G$ has the AP if and only if $\Gamma$ has the AP. Also, the AP is known to pass to closed subgroups and to extensions, the latter meaning that if $N$ is a closed normal subgroup of a locally compact group $G$ such that both $N$ and $G / N$ have the AP, then $G$ has the AP.

\subsection{Universal covering groups} \label{subsec:universalcoveringgroups}
Consider a connected Lie group $G$. A covering group of $G$ is a Lie group $\widetilde{G}$ together with a surjective Lie group homomorphism $\sigma:\widetilde{G} \rightarrow G$ such that $(\widetilde{G},\sigma)$ is a topological covering space of $G$. If a covering space is simply connected, it is called a universal covering space. Every connected Lie group $G$ has a universal covering space $\widetilde{G}$, and it can be made into a Lie group. Indeed, let $\sigma:\widetilde{G} \rightarrow G$ be the corresponding covering map, and let $\tilde{1} \in \sigma^{-1}(1)$. Then there exists a unique multiplication on $\widetilde{G}$ that makes $\widetilde{G}$ into a Lie group in such a way that $\sigma$ is a surjective Lie group homomorphism. The group $\widetilde{G}$ is called a universal covering group of $G$. The universal covering groups of connected Lie groups are unique up to isomorphism. For more details, we refer to \cite[Section I.11]{knapp}.

\subsection{Polar decomposition of Lie groups} \label{subsec:kakdecomposition}
Let $G$ be a connected semisimple Lie group with Lie algebra $\mathfrak{g}$. Then $G$ has a polar decomposition $G=KAK$, where $K$ arises from a Cartan decomposition $\mathfrak{g}=\mathfrak{k} + \mathfrak{p}$ on the Lie algebra level ($K$ has Lie algebra $\mathfrak{k}$), and $A$ is an abelian Lie group such that its Lie algebra $\mathfrak{a}$ is a maximal abelian subspace of $\mathfrak{p}$. If $G$ has finite center, then the group $K$ is a maximal compact subgroup. The dimension of the Lie algebra $\mathfrak{a}$ is called the real rank of $G$. In general, if we have a polar decomposition $G=KAK$, it is not the case that for $g \in G$ there exist unique $k_1,k_2 \in K$ and $a \in A$ such that $g=k_1ak_2$. However, by choosing a set of positive roots and restricting to the closure $\overline{A^{+}}$ of the positive Weyl chamber $A^{+}$, we still have a decomposition $G=K\overline{A^{+}}K$. Moreover, if $g=k_1ak_2$, where $k_1,k_2 \in K$ and $a \in \overline{A^{+}}$, then $a$ is unique. We also use the terminology polar decomposition for such a $K\overline{A^{+}}K$ decomposition. For details, see \cite[Section IX.1]{helgasonlie}.

\subsection{The Levi decomposition of a Lie group} \label{subsec:levidecomposition}
Let $G$ be a connected Lie group with Lie algebra $\mathfrak{g}$, and let $\mathfrak{r}$ be the solvable radical of $\mathfrak{g}$. The Lie algebra $\mathfrak{g}$ decomposes as $\mathfrak{g}=\mathfrak{r} \rtimes \mathfrak{s}$, where $\mathfrak{s}$ is a semisimple subalgebra. We can write $\mathfrak{s}=\mathfrak{s}_1 \oplus \cdots \oplus \mathfrak{s}_n$, where the $\mathfrak{s}_i$'s, with $i=1,\ldots,n$, denote the simple summands of $\mathfrak{s}$. Let $R$, $S$ and $S_i$, with $i=1,\ldots,n$, denote the corresponding analytic subgroups of $G$. In this way, $S$ becomes a maximal semisimple analytic subgroup of $G$, locally isomorphic to the product $S_1\times\cdots\times S_n$ of simple factors.

The subgroup $R$ is solvable, normal and closed, but $S$ need not be closed. Also, $G=RS$, but $R \cap S$ may contain more than one element. However, if $G$ is simply connected, the subgroup $S$ is actually closed and the Levi decomposition becomes a semi-direct product $G=R \rtimes S$. For details, we refer to \cite[Section 3.18]{varadarajan}.

\section{Invariant means associated with locally compact groups} \label{sec:invariantmeans}
In this section, we first recall some well-known facts on invariant means associated with locally compact groups. After that, we prove Theorem \ref{thm:uniqueinvariantmeanm0a}. Let $G$ be a locally compact group, and let $X$ be a subspace of $L^{\infty}(G)$ that is closed under complex conjugation and that contains the constant functions. Recall that a mean $m$ on $X$ is a positive linear functional $m \colon X \rightarrow \mathbb{C}$ satisfying $\|m\|=m(1)=1$. Let $L_g$ and $R_g$ denote the left and right translation operator associated with an element $g \in G$, respectively. More precisely, for a function $\varphi \colon G \rightarrow \mathbb{C}$, we have $(L_g\varphi)(h)=\varphi(g^{-1}h)$ and $(R_g\varphi)(h)=\varphi(hg)$ for all $g,h \in G$. A mean $m$ on $X$ is said to be left invariant if $m(L_g\varphi)=m(\varphi)$ for all $\varphi \in X$ and $g \in G$. Right invariant means are defined similarly. A two-sided invariant mean is a mean that is both left and right invariant.

Let $C_b(G)$ denote the space of continuous bounded functions on $G$. Recall that a function $\varphi\in C_b(G)$ is left uniformly continuous if $\|L_g\phi - \phi\|_\infty\to 0$ as $g\to 1$. Right uniform continuity is defined analogously, using right translation operators. Note that the definitions of left and right uniform continuity are often reversed in the literature. Let $\mathrm{LUC}_b(G)$ and $\mathrm{RUC}_b(G)$ denote the left and right uniformly continuous bounded functions on $G$, respectively, and let $\mathrm{UC}_b(G)$ denote the two-sided uniformly continuous bounded functions on $G$. We have $\mathrm{UC}_b(G) \subset \mathrm{LUC}_b(G) \subset C_b(G) \subset L^{\infty}(G)$, where each of the inclusions is a (two-sided) translation invariant $\|.\|_{\infty}$-closed inclusion.

Recall that $G$ is amenable if there exists a left invariant mean on $L^{\infty}(G)$. It is well-known (see, e.g., \cite[Theorem 2.2.1]{greenleaf} or \cite[Theorem 1.1.9]{runde}) that $G$ is amenable if and only if there exists a left invariant mean on $C_b(G)$ (or, equivalently, on $\mathrm{LUC}_b(G)$, $\mathrm{RUC}_b(G)$ or $\mathrm{UC}_b(G)$). The previous assertions can equivalently be formulated for right invariant and two-sided invariant means.

We now turn to the notion of weakly almost periodic function. The study of such functions was initiated by Eberlein \cite{eberlein} for locally compact abelian groups and generalized to locally compact (semi)groups by de~Leeuw and Glicksberg \cite{deleeuwglicksberg1}, \cite{deleeuwglicksberg2} (see also \cite{burckel}).
\begin{dfn}
  A function $\varphi \in C_b(G)$ is weakly almost periodic if its orbit $O(\varphi,L)=\{L_g\varphi \mid g \in G\}$ under left translations (equivalently, its orbit $O(\varphi,R)$ under right translations) is relatively compact in the weak topology on $C_b(G)$.
\end{dfn}
The space of weakly almost periodic functions on $G$ is denoted by $W(G)$. For a function $\varphi\colon G \rightarrow \mathbb{C}$ defined on a group $G$, let $\widecheck \varphi(g) = \varphi(g^{-1})$. The space $W(G)$ is a closed translation and inversion invariant (i.e., if $\varphi \in W(G)$, then $\widecheck \varphi \in W(G)$ as well) subalgebra of $C_b(G)$ containing the constant functions. It is well-known that $W(G) \subset \mathrm{UC}_b(G)$ (see \cite[Theorem 3.11]{burckel}).

Let $C(\varphi,L)$ denote the weakly closed convex hull of $O(\varphi,L)$, which coincides with the norm closed convex hull of $O(\varphi,L)$, and let $C(\varphi,R)$ be defined analogously. The following powerful result follows from the Ryll-Nardzewski fixed-point theorem (see, e.g., \cite[Appendix 2]{greenleaf}).
\begin{thm} \label{thm:uniqueinvariantmeanwap} (See \cite[\S 3.1]{greenleaf} or \cite[Theorem 1.25]{burckel})
  Let $G$ be a locally compact group. Then there exists a unique left invariant mean $m$ on $W(G)$. Explicitly, given a function $\varphi \in W(G)$, its mean $m(\varphi)$ is given by the unique constant in $C(\varphi,L)$, which equals the unique constant in $C(\varphi,R)$. Moreover, $m$ is right invariant and invariant under inversion, i.e., $m(\widecheck{\varphi})=m(\varphi)$ for all $\varphi \in W(G)$.
\end{thm}
As mentioned in \cite[Corollary 3.7]{burckel}, if $G$ is a locally compact group, then $C_0(G) \subset W(G)$. Moreover, if $G$ is non-compact and $\varphi \in C_0(G)$, then $m(\varphi)=0$.

It is well-known that the Fourier-Stieltjes algebra $B(G)$ (see Section \ref{subsec:fsf}) of $G$ is a translation and inversion invariant unital subalgebra of $W(G)$. In fact, the same holds for the space $M_0A(G)$ of completely bounded Fourier multipliers (see Section \ref{subsec:fm}), which contains $B(G)$. This is mentioned in \cite[p.~2527]{xu} and \cite[p.~58]{bozejko}, but we include a (different) proof for completeness.
\begin{prp} \label{prp:cbfminwap}
  The algebra $M_0A(G)$ is a translation and inversion invariant unital subalgebra of $W(G)$.
\end{prp}
\begin{proof}
We only prove that $M_0A(G)\subset W(G)$, leaving the other assertions to the reader. Let $\varphi \in M_0A(G)$. It is well-known that $\varphi \in C_b(G)$. By the characterization of $M_0A(G)$ by Bo\.zejko and Fendler (see Section \ref{subsec:fm}), there exist bounded continuous maps $P,Q \colon G \rightarrow \mathcal{H}$ from $G$ to a Hilbert space $\mathcal{H}$ such that for all $g,h \in G$, we have $\varphi(h^{-1}g)=\langle P(g),Q(h) \rangle$.

For each $\xi\in\mathcal H$, let $T(\xi)$ be the function on $G$ defined by $T(\xi)(h) = \langle\xi,Q(h^{-1})\rangle$ for $h\in G$. Clearly, $T(\xi)\in C_b(G)$, and $T\colon\mathcal H\to C_b(G)$ is a bounded linear map with $\|T\|\leq\|Q\|_\infty$. Hence, $T$ is also weak-weak continuous.

The set $B = \{\xi\in\mathcal H \mid \|\xi\|\leq\|P\|_\infty\}$ is weakly compact in $\mathcal H$, and hence $T(B)$ is weakly compact in $C_b(G)$. It is easy to check that $R_g\phi = T(P(g))$, and hence $O(\phi,R)$ is a subset of $T(B)$. This shows that $O(\phi,R)$ is relatively weakly compact in $C_b(G)$, which implies that $\phi \in W(G)$.
\end{proof}
\begin{cor}\label{cor:UC}
Completely bounded Fourier multipliers are uniformly continuous.
\end{cor}
\begin{rmk}
It is not true, in general, that Fourier multipliers that are not completely bounded are weakly almost periodic \cite[Proposition 5.2.2]{bozejko}.
\end{rmk}
The following theorem implies Theorem \ref{thm:uniqueinvariantmeanm0a}. It also implies the existence of a unique left invariant mean on $B(G)$. However, this was already known from the work of Godement \cite[Chapitre III]{godement}.
\begin{thm} \label{thm:restrictionofmeanonwap}
Let $G$ be a locally compact group, and let $X$ be a subspace of $W(G)$ that is closed under conjugation and left translation and that contains the constant function $1$. Then there exists a unique left invariant mean $M$ on $X$. If $X$ is, additionally, closed under right translation, then $M$ is also right invariant. Moreover, if $X$ is closed under inversion, then $M$ is inversion invariant.
\end{thm}
\begin{proof}
The existence of a left invariant mean $M$ follows from Theorem \ref{thm:uniqueinvariantmeanwap}. Suppose now that $M^{\prime}$ is a left invariant mean, and let $\varphi \in X$ be given. By invariance, $M^{\prime}$ is constant on the set $\{L_g \varphi \mid g\in G\}$ of left translates of $\varphi$, and by linearity, $M^{\prime}$ is constant on the convex hull $\mathrm{conv}\{L_g \varphi \mid g\in G\}$. By Theorem \ref{thm:uniqueinvariantmeanwap}, there exists a unique constant function $M(\varphi)$ in the norm closure of $\mathrm{conv}\{L_g \varphi \mid g\in G\}$.

Recall that if $(S,\mu)$ is a measure space and $X$ is a subspace of $L^\infty(S,\mu)$ such that $X$ is closed under conjugation and such that $1 \in X$, then any mean on $X$ is bounded with respect to the uniform norm (see \cite[Proposition 3.2]{pier}). It follows that $M^{\prime}$ is continuous. Hence, $M^{\prime}(\varphi) = M^{\prime}(M(\varphi)) = M(\varphi)$. This shows that $M^{\prime} = M$, which proves uniqueness. The additional assertions follow directly from Theorem \ref{thm:uniqueinvariantmeanwap}.
\end{proof}
\begin{proof}[Proof of Theorem \ref{thm:uniqueinvariantmeanm0a}]
As proved in Proposition \ref{prp:cbfminwap}, the space $M_0A(G)$ is a translation and inversion invariant unital subalgebra of $W(G)$. The result now follows from Theorem \ref{thm:restrictionofmeanonwap}.
\end{proof}

\section{A characterization of property (T)} \label{sec:propertyt}
In this section, we present a characterization of property (T) based on the unique left invariant mean on the Fourier-Stieltjes algebra. This characterization is essentially due to Akemann and Walter \cite[Lemma 2]{akemannwalter} (see also \cite[Lemma 3.1]{valette}), but since our formulation is different, because it explicitly involves the mean on the Fourier-Stieltjes algebra, we give a proof.

Let $G$ be a locally compact group, and let $dg$ be a Haar measure on $G$. Let $C^{\ast}(G)$ denote the universal $C^{\ast}$-algebra of $G$. Recall that, for $f\in L^1(G)$ and $\varphi \in L^{\infty}(G)$, we set $\langle f,\varphi \rangle = \int f(g)\varphi(g)dg$. If $\phi\in B(G)$, then $\varphi$ can be viewed as a functional on $L^1(G)$ that extends by continuity to $C^{\ast}(G)$. Moreover, the Fourier-Stieltjes algebra $B(G)$ is linearly isomorphic to the dual $C^{\ast}(G)^*$ of $C^{\ast}(G)$, and in this way $B(G)$ carries a weak$^*$ topology. Note from Theorem~\ref{thm:restrictionofmeanonwap} that $B(G)$ carries a unique left invariant mean.

\begin{prp}\label{prp:propertyT}
For a locally compact group $G$, the following are equivalent:
\begin{enumerate}[(i)]
	\item the group $G$ has property (T),
	\item the unique left invariant mean $m$ on $B(G)$ is weak$^*$ continuous. In other words, $m\in C^*(G)$.
\end{enumerate}
\end{prp}
Before proving this proposition, we recall some well-known facts.
Let $\pi_U$ denote the universal representation of both $G$ and $L^1(G)$. Then $\pi_U(L^1(G))$ is norm dense in $C^{\ast}(G)$ and weak operator dense in the von Neumann algebra $C^{\ast}(G)^{**}$. Also, the span of $\pi_U(G)$ is weak operator dense in $C^{\ast}(G)^{**}$.

The representation $1_G$ of $L^1(G)$ coming from the trivial representation of $G$ is given by $1_G(f) = \int f(g)dg$. This representation extends to a normal $\ast$-representation $1_G\colon C^{\ast}(G)^{**} \rightarrow \mathbb{C}$.

The following identities can easily be shown to hold for all $g \in G$, $\varphi \in B(G)$ and $a\in L^1(G)$, and hence also by continuity for $a \in C^{\ast}(G)^{**}$:
\begin{align}
\label{eq:left-translation}
\langle \pi_U(g) a,\varphi\rangle &= \langle a, L_{g^{-1}} \varphi\rangle,
\end{align}
\begin{align}
\label{eq:right-translation}
\langle a\pi_U(g),\varphi\rangle &= \langle a, R_g \varphi \rangle,
\end{align}
\begin{align}
\label{eq:one}
\langle a,1\rangle &= 1_G(a),
\end{align}
\begin{align}
\label{eq:involution}
\langle a, \widetilde\varphi\rangle &= \overline{\langle a^*,\varphi\rangle},
\end{align}
where $\widetilde{\varphi}(x)=\overline{\varphi(x^{-1})}$. Let $p$ denote the central cover of $1_G$, i.e., $p$ is a central minimal projection in $C^{\ast}(G)^{**}$ such that $\ker 1_G = (1-p)C^{\ast}(G)^{**}$. The following result is contained in \cite[Lemma 2]{akemannwalter}.
\begin{lem} \label{lem:valette}
For a locally compact group $G$, the following are equivalent:
\begin{enumerate}[(i)]
	\item the group $G$ has property (T),
	\item the central cover of the trivial representation $1_G\colon C^{\ast}(G)^{**} \rightarrow \mathbb{C}$ belongs to $C^{\ast}(G)$.
\end{enumerate}
\end{lem}
We can now relate this result to the unique left invariant mean on $B(G)$.
\begin{lem}\label{lem:pm}
Let $G$ be a locally compact group, let $p$ be the central cover of the trivial representation $1_G\colon C^{\ast}(G)^{**} \rightarrow \mathbb{C}$, and let $m$ be the unique left invariant mean on $B(G)$. If $\Phi\colon B(G)^{\ast} \rightarrow C^{\ast}(G)^{**}$ denotes the canonical isomorphism, then $\Phi(m) = p$.
\end{lem}
\begin{proof}
Put $q = \Phi(m)$. Recall that $m$ is also right invariant and inversion invariant. We know that $m(\overline{\varphi}) = \overline{m(\varphi)}$ for every $\varphi\in B(G)$. Since also $m(\widecheck\varphi) = m(\varphi)$, we obtain, using \eqref{eq:involution}, that $\langle q^*,\varphi \rangle = \overline{\langle q,\widetilde\varphi \rangle} = \langle q,\widecheck \varphi \rangle = \langle q,\varphi \rangle$ for every $\varphi \in B(G)$. This shows that $q^* = q$.

As $m(1) = 1$, we see by \eqref{eq:one} that $1_G(q) = 1$. Also, by \eqref{eq:left-translation} and \eqref{eq:right-translation}, left and right invariance of $m$ translates into the fact that $\pi_U(g)q = q = q\pi_U(g)$ for every $g\in G$. Hence, for every $a\in C^{\ast}(G)^{**}$, we have $aq = 1_G(a)q = qa$. In particular, we see that $q^2 = q$ and $qp = q = pq$. It now follows that $q$ is a non-zero subprojection of $p$, and by minimality of $p$ we conclude that $q=p$.
\end{proof}
\begin{proof}[Proof of Proposition \ref{prp:propertyT}]
The weak$^*$ continuous functionals in $C^{\ast}(G)^{**}$ are exactly the elements of $C^{\ast}(G)$. The proposition now follows by combining Lemmas \ref{lem:valette} and \ref{lem:pm}.
\end{proof}

\section{Property \texorpdfstring{(T$^{\ast}$)}{(T*)}} \label{sec:propertytstar}
In this section, we introduce and study a strengthening of property (T), which turns out to be an obstruction to the AP. Recall the following two norms on $L^1(G)$. For $f \in L^1(G)$, we set
\begin{alignat*}{3}
&\|f\|_{C^{\ast}(G)} &= &\sup\{|\langle f,\varphi\rangle| \mid \varphi\in B(G),\ \|\varphi\|_{B(G)}\leq 1\}, \\
&\|f\|_{M_0A(G)_*} &= &\sup\{|\langle f,\varphi\rangle| \mid \varphi\in M_0A(G),\ \|\varphi\|_{M_0A(G)}\leq 1\}.
\end{alignat*}
The second norm was already used in Section \ref{subsec:ap}.
\begin{dfn}
Let $G$ be a locally compact group, and let $m$ be the unique left invariant mean on $M_0A(G)$. The group $G$ is said to have property (T$^{\ast}$) if $m$ is a weak$^*$ continuous functional.
\end{dfn}

We sometimes write $m_G$ instead of $m$, if we want to emphasize the group $G$. From the next lemma, it follows that property (T$^{\ast}$) is indeed a strengthening of property (T). Recall that $B(G)\subset M_0A(G)$.
\begin{lem}
The weak$^*$ topology on $B(G)$ is stronger than the restriction of the weak$^*$ topology on $M_0A(G)$ to $B(G)$.
\end{lem}
\begin{proof}
The identity map $L^1(G) \rightarrow L^1(G)$ extends to a contractive linear map $M_0A(G)_* \rightarrow C^{\ast}(G)$. Its Banach space adjoint is then a weak$^*$-weak$^*$ continuous contraction from $B(G)$ to $M_0A(G)$ that is easily seen to be the inclusion map.
\end{proof}
\begin{prp}
For locally compact groups, property (T$^*$) implies property (T).
\end{prp}
\begin{proof}
This follows from the characterization of property (T) given in Proposition~\ref{prp:propertyT} together with the previous lemma.
\end{proof}
\begin{prp}
Compact groups have property (T$^{\ast}$).
\end{prp}
\begin{proof}
For a compact group $G$, the map $\phi\mapsto \langle 1,\phi\rangle$ clearly defines a weak$^*$ continuous two-sided invariant mean on $M_0A(G)$.
\end{proof}
The following result shows that property (T$^{\ast}$) is an obstruction to the AP.
\begin{prp}\label{prp:obstruction}
  Let $G$ be a locally compact group with the AP and property (T$^{\ast}$). Then $G$ is compact.
\end{prp}
\begin{proof}
  Suppose that $G$ has the AP and property (T$^{\ast}$). Then there is a net $(\varphi_{\alpha})$ in $A(G)$ such that $\varphi_{\alpha} \to 1$ in the weak* topology on $M_0A(G)$. Moreover, since $m(1)=1$, we have $\lim_{\alpha}m(\varphi_{\alpha})=1$. However, we know that $A(G) \subset C_0(G)$ and $m(\psi)=0$ for all $\psi \in C_0(G)$ if $G$ is non-compact. This leads to a contradiction. Hence, $G$ has to be compact.
\end{proof}
\begin{rmk}
  Property (T$^{\ast}$) is strictly stronger than property (T). Indeed, it is well-known that the non-compact group $\mathrm{Sp}(n,1)$ has property (T), but since it has the AP (because it is weakly amenable by \cite{cowlinghaagerup}), it does not have property (T$^{\ast}$).
\end{rmk}
The following lemma is a technical result that we use in what follows, the proof of which is left to the reader.
\begin{lem}\label{lem:w-star}
Let $X$ and $Y$ be Banach spaces, and let $T\colon Y^*\to X^*$ be a bounded linear operator between their dual spaces. Suppose there exist a dense subspace $X_0 \subset X$ and a linear operator $S\colon X_0\to Y$ such that $\langle Sx,y\rangle = \langle x,Ty\rangle$ for every $x\in X_0$ and $y\in Y^*$. Then $S$ is bounded and $S^* = T$. In particular, $T$ is weak$^*$--weak$^*$ continuous.
\end{lem}
Recall that the involution $f\mapsto f^*$ on $L^1(G)$ is defined by $f^*(x) = \Delta(x^{-1})\overline{f(x^{-1})}$, where $\Delta$ denotes the modular function of $G$.
\begin{lem}\label{lem:convolution}
Let $\pi\colon H\to G$ be a continuous group homomorphism of locally compact groups. If $u\in C_c(G)$ is non-negative with $\|u\|_1 = 1$ and $u = u^*$, then the linear map $T\colon M_0A(G)\to M_0A(H)$ given by
$$
T\phi = (u*\phi)\circ\pi
$$
is weak$^*$--weak$^*$ continuous.
\end{lem}
\begin{proof}
For $h\in L^1(H)$, the map
$$
F \colon f\mapsto \int_H (f\circ\pi) h \ d\mu_H, \quad f\in C_0(G)
$$
is a linear functional on $C_0(G)$ of norm at most $\|h\|_1$. Hence, there exists a unique Radon measure on $G$, denoted by $h\mu_H$, such that the functional $F$ is given by integration with respect to this measure.

Consider the map $S\colon L^1(H)\to L^1(G)$ defined by
$$
S(h) = u*h\mu_H, \quad h\in L^1(H).
$$
A simple computation shows that for every $h\in L^1(H)$ and $\phi\in M_0A(G)$,
\begin{align*}
\langle u*h\mu_H, \phi\rangle =  \langle h, (u*\phi)\circ\pi\rangle.
\end{align*}
Here we have used that $u = u^*$. The map $T$ is clearly a contraction. Hence, it follows from Lemma \ref{lem:w-star} that $S$ extends to a contraction $M_0A(H)_*\to M_0A(G)_*$ with adjoint $T$.
\end{proof}
The following permanence property of property (T$^{\ast}$) plays an important role in the characterization of connected Lie groups with the AP in Section \ref{sec:connectedliegroups}.
\begin{prp} \label{prp:denseimage}
  Let $G$ and $H$ be locally compact groups, and let $\pi\colon H \rightarrow G$ be a continuous group homomorphism with dense image. If $H$ has property (T$^{\ast}$), then $G$ has property (T$^{\ast}$).
\end{prp}
\begin{proof}
Choose $g\in C_c(G)$ as in Lemma \ref{lem:convolution}, and let $T\colon M_0A(G)\to M_0A(H)$ be the weak$^*$--weak$^*$ continuous linear map from Lemma~\ref{lem:convolution}. Consider the functional on $M_0A(G)$ given by $m(\phi) = m_H(T\phi)$, where $m_H$ is the unique left invariant mean on $M_0A(H)$. We claim that $m$ is the unique left invariant mean $m_G$ on $M_0A(G)$. By the results of Section \ref{sec:invariantmeans}, it suffices to prove that $m$ is a right invariant mean. It is straightforward to see that $m$ is a mean. For $x\in H$ and $\phi\in M_0A(G)$, it is easy to verify that
$$
R_x (T\phi) = T(R_{\pi(x)}\phi).
$$
Since $m_H$ is right invariant, it follows that $m$ is $\pi(H)$-right invariant. As $m$ is a mean on $M_0A(G)$, it is continuous with respect to the uniform topology (\cite[Proposition 3.2]{pier}). Since any $\phi\in M_0A(G)$ is uniformly continuous (Corollary~\ref{cor:UC}) and $\pi(H)$ is dense in $G$, we conclude that $m$ is $G$-right invariant, and hence $m_G = m = m_H \circ T$.

If $H$ has property (T$^*$), then $m_H$ is weak$^*$ continuous. Since $T$ is weak$^*$--weak$^*$ continuous, it follows that $m_G$ is weak$^*$ continuous.
\end{proof}
\begin{cor} \label{cor:quotients}
Property (T$^{\ast}$) is inherited by quotients: if $G$ has property (T$^*$), then so does $G/N$ for every closed normal subgroup $N$ of $G$.
\end{cor}
\begin{lem}
If $G$ is a locally compact group with property (T$^*$), then there is a sequence $f_n\in L^1(G)$ with $f_n\geq 0$ and $\|f_n\|_1 = 1$ for all $n \in \mathbb{N}$ such that $\|f_n - m\|_{M_0A(G)_*} \to 0$. Here, $m\in M_0A(G)_*$ is the left invariant mean on $M_0A(G)$.
\end{lem}
\begin{proof}
Recall that $L^1(G)$ is the set of normal functionals on $L^\infty(G)$, and that the normal states are weak$^*$ dense in the state space of $L^\infty(G)$. We may extend the mean $m\in W(G)^*$ to a state on $L^\infty(G)$ and then obtain a net $f_\alpha\in L^1(G)$ such that $f_\alpha\geq 0$, $\|f_\alpha\|_1 = 1$ and $f_\alpha \to m$ in the weak$^*$ topology on $L^\infty(G)^*$. In particular, $f_\alpha\to m$ weakly inside $M_0A(G)_*$. Passing to convex combinations, we may arrange that $f_\alpha$ converges to $m$ in norm inside $M_0A(G)_*$. Clearly, then also a sequence with the desired properties exists.
\end{proof}
\begin{prp}
  Let $G_1$ and $G_2$ be two locally compact groups. The direct product $G=G_1 \times G_2$ has property (T$^{\ast}$) if and only if $G_1$ and $G_2$ have property (T$^{\ast}$).
\end{prp}
\begin{proof}
Suppose that $G_1 \times G_2$ has property (T$^{\ast}$). From Corollary \ref{cor:quotients}, it follows that $G_1$ and $G_2$ have property (T$^{\ast}$).

Suppose that $G_1$ and $G_2$ have property (T$^{\ast}$). For $i=1,2$, let $m_i$ denote the left invariant mean on $M_0A(G_i)$. By the previous lemma, there exist sequences $f_n^{(i)} \in L^1(G_i)^{+}$ such that $\int f_n^{(i)}d\mu_{G_i}=1$ and $\|f_n^{(i)}-m_i\|_{M_0A(G_i)_{*}} \to 0$ for $n \to \infty$ and $i=1,2$. Since $m_i$ is left invariant on $M_0A(G_i)$, it follows that
$$
\|L_{g_i} f_n^{(i)} - f_n^{(i)}\|_{M_0A(G_i)_*}\to 0 \qquad\text{for all } g_i\in G_i.
$$
From \cite[Lemma~1.4]{cowlinghaagerup} we know that whenever $f_i\in L^1(G_i)$, then
\begin{align}\label{eq:crossnorm}
\|f_1\times f_2\|_{M_0A(G_1\times G_2)_*} = \|f_1\|_{M_0A(G_1)_*}\|f_2\|_{M_0A(G_2)_*}.
\end{align}
It follows from this that $f_n^{(1)} \times f_n^{(2)}$ is a Cauchy sequence in $M_0A(G)_*$. Let
$$
M=\lim_{n \to \infty} f_n^{(1)} \times f_n^{(2)}.
$$
Then $M$ is a weak$^*$ continuous mean on $M_0A(G)$. It follows from \eqref{eq:crossnorm} together with the triangle inequality that
$$
\|L_g (f_n^{(1)}\times f_n^{(2)}) - f_n^{(1)}\times f_n^{(2)}\|_{M_0A(G_1\times G_2)_*} \to 0 \qquad\text{for all } g\in G_1\times G_2.
$$
Therefore $M$ is left invariant, and hence $m_G = M$ is weak$^*$ continuous.
\end{proof}
\begin{prp}\label{prp:modulo-compact}
Let $G$ be a locally compact group with a compact normal subgroup $K$. Then $G$ has property (T$^*$) if and only if $G/K$ has property (T$^*$).
\end{prp}
\begin{proof}
One direction follows from Corollary~\ref{cor:quotients}. In the other direction, observe first that the map $T\colon M_0A(G)\to M_0A(G/K)$ defined by
$$
(T\phi)(xK) = \int_K \phi(xk) dk
$$
is well-defined and weak$^*$--weak$^*$ continuous. Here $dk$ denotes normalized Haar measure on $K$. When $x\in G$ and $\phi\in M_0A(G)$, then one can readily check that
$$
T(L_x\phi) = L_{\dot x}(T\phi),
$$
where $\dot x\in G/K$ is the image of $x$. So by left invariance of $m_{G/K}$,
$$
(m_{G/K}\circ T)(L_x\phi) = m_{G/K}(L_{\dot x}(T\phi)) = m_{G/K}(T\phi).
$$
Hence, $m_{G/K} \circ T$ defines a left invariant mean on $M_0A(G)$ so that $m_{G/K}\circ T = m_G$. It is now clear that if $G/K$ has property (T$^*$), then so does $G$. 
\end{proof}

The following lemma will be relevant in the next section, where we establish property (T$^*$) for some specific groups. For a locally compact group $G$ with a compact subgroup $K$ we define, for $\phi\in C(G)$ or $\phi\in L^1(G)$,
\begin{equation} \label{eq:biinv}
\phi^K(g) = \int_K\int_K \phi(k_1gk_2) dk_1dk_2, \qquad g\in G,
\end{equation}
where $dk_1$ and $dk_2$ are normalized Haar measures on $K$.

\begin{lem}\label{lem:average}
Let $G$ be a locally compact group with a compact subgroup $K$. For any $\phi\in W(G)$, we have $m(\phi) = m(\phi^K)$.
\end{lem}
\begin{proof}
Since $\phi$ is left and right uniformly continuous, the function $\phi^K$ is a Banach space integral in $W(G)$ with the uniform norm (see \cite[Theorem~3.27]{rudin}):
$$
\phi^K = \int_K\int_K L_{k_1^{-1}} R_{k_2} \phi dk_1dk_2.
$$
As $m\in W(G)^*$ and $m$ is two-sided invariant, we have
$$
m(\phi^K) = \int_{K\times K} m(L_{k_1^{-1}} R_{k_2} \phi) dk_1dk_2 = \int_{K\times K} m(\phi) dk_1dk_2 = m(\phi).
$$
\end{proof}

\section{Three groups with property \texorpdfstring{(T$^{\ast}$)}{(T*)}} \label{sec:sl3sp2sp2cov}
In this section, we prove Theorem \ref{thm:sl3sp2sp2cov}. Hereto, we study the groups $\mathrm{SL}(3,\mathbb{R})$, $\mathrm{Sp}(2,\mathbb{R})$ and $\widetilde{\mathrm{Sp}}(2,\mathbb{R})$ separately. The first two groups are easier to handle, because they have finite center, so that we can use the following powerful result of Veech \cite[Theorem 1.4]{veech}.
\begin{thm}[Veech] \label{thm:veech}
Let $G$ be a non-compact connected simple Lie group with finite center. Then $W(G) = C_0(G) \oplus \mathbb C 1$ and
$$
m(\phi) = \lim_{g\to\infty} \phi(g)
$$
for every $\phi\in W(G)$.
\end{thm}
Let $G$ be a locally compact group with a subgroup $H$. A function $\varphi \colon G \rightarrow \mathbb{C}$ is said to be $H$-bi-invariant if $\varphi(h_1gh_2)=\varphi(g)$ for all $g \in G$ and $h_1,h_2 \in H$, and it is said to be $\mathrm{Int}(H)$-invariant if $\varphi(hgh^{-1})=\varphi(g)$ for all $g \in G$ and $h \in H$.
\begin{thm} \label{thm:sl3}
The group $\mathrm{SL}(3,\mathbb{R})$ has property (T$^{\ast}$).
\end{thm}
\begin{proof}
Let $G=\mathrm{SL}(3,\mathbb{R})$, and let $K=\mathrm{SO}(3)$ be its maximal compact subgroup. By (the proof of) \cite[Lemma 2.5]{haagerupdelaat1}, the averaging map $\varphi \mapsto \varphi^K$ (see \eqref{eq:biinv}) on $M_0A(G)$ is weak$^*$--weak$^*$ continuous. By Lemma~\ref{lem:average}, $m(\phi) = m(\phi^K)$, and hence it suffices to prove that the restriction $m'$ of $m$ to the $K$-bi-invariant completely bounded Fourier multipliers $M_0A(K\backslash G/K)$ is weak$^*$ continuous.

Let $\phi\in M_0A(G)$. By the aforementioned result of Veech (Theorem \ref{thm:veech}), the limit $\lim_{g\to\infty} \phi(g)$ exists and agrees with the mean $m(\varphi)$ of $\varphi$. Hence, the kernel of $m'$ is precisely the subspace $M_0A(K\backslash G/K) \cap C_0(G)$. It is proved in \cite[p.~957]{haagerupdelaat1} that this subspace is weak$^*$ closed, and hence $m'$ is weak$^*$ continuous.
\end{proof}
\begin{rmk}
It is, in fact, not necessary to use Veech's result. The necessary consequence of this theorem is that the elements of $M_0(K \backslash G / K)$ converge if we take the limit in $G$ to infinity. This fact follows from \cite[p.~957]{haagerupdelaat1}, which in particular establishes the value of the mean at elements of $M_0A(K \backslash G / K)$. In Theorem \ref{thm:sp2-covering}, we will establish property (T$^{\ast}$) for $\widetilde{\mathrm{Sp}}(2,\mathbb{R})$, in which case we cannot use Veech's result (because $\widetilde{\mathrm{Sp}}(2,\mathbb{R})$ has infinite center), so in that case we will provide the argument just mentioned.
\end{rmk}
\begin{cor}
For $G = \SL(3,\mathbb R)$, the subspace $M_0A(G)\cap C_0(G)$ is weak$^*$ closed in $M_0A(G)$.
\end{cor}
\begin{proof}
By Theorem \ref{thm:veech}, we have that $M_0A(G)\cap C_0(G)$ is precisely the kernel of $m\in M_0A(G)_*$.
\end{proof}

\begin{rmk}\label{rmk:cover-SL3}
The universal covering group $\widetilde\SL(3,\mathbb R)$ of $\SL(3,\mathbb R)$ has finite center. It follows from Theorem~\ref{thm:sl3} and Proposition~\ref{prp:modulo-compact} that $\widetilde\SL(3,\mathbb R)$ has property (T$^*$).
\end{rmk}

Let $I_2$ denote the $2 \times 2$ identity matrix, and let $J$ be the matrix defined by
\[
  J=\left( \begin{array}{cc} 0 & I_2 \\ -I_2 & 0 \end{array} \right).
\]
Recall that the symplectic group $\mathrm{Sp}(2,\mathbb{R})$ is defined by
\[
	\mathrm{Sp}(2,\mathbb{R})=\{g \in \mathrm{GL}(4,\mathbb{R}) \mid g^T J g = J\}.
\]
Here, $g^T$ denotes the transpose of $g$. Let $K$ be the maximal compact subgroup of $\Sp(2,\bbR)$ given by
\[
  K= \bigg\{ \left( \begin{array}{cc} A & -B \\ B & A \end{array} \right) \in \mathrm{M}_{4}(\mathbb{R}) \biggm\vert A+iB \in \mathrm{U}(2) \bigg\}.
\]
This group is isomorphic to $\mathrm{U}(2)$. A polar decomposition of $\Sp(2,\bbR)$ is given by $\Sp(2,\bbR)=KAK$, where
\[
	A=\left\{D(\beta,\gamma)= \left( \begin{array}{cccc} e^{\beta} & 0 & 0 & 0 \\ 0 & e^{\gamma} & 0 & 0 \\ 0 & 0 & e^{-\beta} & 0 \\ 0 & 0 & 0 & e^{-\gamma} \end{array} \right) \Biggm\vert \beta,\gamma \in \mathbb{R}\right\}.
\]
By restricting to the positive Weyl chamber (see Section \ref{subsec:kakdecomposition}), we obtain the decomposition $G=K\overline{A^{+}}K$, where
\[
	\overline{A^{+}}=\{D(\beta,\gamma) \in A \mid \beta \geq \gamma \geq 0\}.
\]
\begin{thm}\label{thm:sp2}
  The group $\mathrm{Sp}(2,\mathbb{R})$ has property (T$^{\ast}$).
\end{thm}
The proof of this theorem is exactly the same as the proof of Theorem \ref{thm:sl3}. The essential part, that the subspace $M_0A(K\backslash G/K)\cap C_0(G)$ is weak$^*$ closed, was shown in \cite[p.~937]{haagerupdelaat1}. Alternatively, Theorem~\ref{thm:sp2} follows from Theorem~\ref{thm:sp2-covering} and Corollary~\ref{cor:quotients}.
\begin{cor}
For $G = \Sp(2,\mathbb R)$, the subspace $M_0A(G)\cap C_0(G)$ is weak$^*$ closed in $M_0A(G)$.
\end{cor}
\begin{thm}\label{thm:sp2-covering}
  The universal covering group $\widetilde{\mathrm{Sp}}(2,\mathbb{R})$ of $\mathrm{Sp}(2,\mathbb{R})$ has property (T$^{\ast}$).
\end{thm}
Before we prove this theorem, we first describe the structure of $\widetilde{\mathrm{Sp}}(2,\mathbb{R})$. For more details on this, we refer to \cite[Section 3]{haagerupdelaat2}.

Let $G=\mathrm{Sp}(2,\mathbb{R})$, and let $\widetilde{G}=\widetilde{\mathrm{Sp}}(2,\mathbb{R})$ be its universal covering group. Let $K$ and $A$ be as above, so that we have $G=KAK$. The image of the Lie algebra of $K$ under the exponential map to $\widetilde{G}$ yields a non-compact subgroup $\widetilde{K}$ of $\widetilde{G}$. Similarly, the image of the Lie algebra of $A$ under the exponential map to $\widetilde{G}$ yields a subgroup $\widetilde{A}$. We have $\widetilde{A} \cong A$, and $\widetilde{G}$ has polar decomposition $\widetilde{G}=\widetilde{K}\widetilde{A}\widetilde{K}$ (see \cite[Section 3]{haagerupdelaat2} for details). The group $\mathrm{SU}(2)$ is a natural subgroup of $\mathrm{U}(2)$. Let $H$ be the image of $\mathrm{SU}(2)$ under this isomorphism, i.e., $H$ is a Lie subgroup of $G$ isomorphic to $\mathrm{SU}(2)$. Similarly, we obtain a subgroup $\widetilde{H}$ of $\widetilde{G}$, which is actually isomorphic to $\mathrm{SU}(2)$ (and hence compact), since $\mathrm{SU}(2)$ is simply connected. 

Let $\mathcal{C}$ be the following class of functions:
\[
  \mathcal{C}=\{\varphi \in C(\widetilde{G}) \mid \varphi \textrm{ is }\widetilde{H}\textrm{-bi-invariant and }\mathrm{Int}(\widetilde{K})\textrm{-invariant}\}.
\]
Consider the generator $\begin{pmatrix} i & 0 \\ 0 & i \end{pmatrix}$ of the Lie algebra of the center of $\mathrm{U}(2)$, and let $Z$ denote the corresponding element of the Lie algebra of $\widetilde{K}$. The elements $\widetilde{v}_t=\widetilde{\exp}(tZ)$ for $t \in \mathbb{R}$ are elements of the center of $\widetilde{K}$. Here $\widetilde{\exp}$ denotes the exponential map associated with $\widetilde{G}$. Moreover, the map $t \mapsto \mathrm{Ad}_{\widetilde{v}_t}$ is periodic.

For $\beta \geq \gamma \geq 0$, consider the element $D(\beta,\gamma)=\diag(e^{\beta},e^{\gamma},e^{-\beta},e^{-\gamma}) \in A$, and let $\widetilde D(\beta,\gamma)$ be its (unique) associated element in $\widetilde{A}$. For $\phi\in M_0A(\widetilde G)$, we put $\dot\phi(\beta,\gamma,t) = \phi(\widetilde v_{\frac t2} \widetilde D(\beta,\gamma))$.

We can now prove the theorem.
\begin{proof}[Proof of Theorem \ref{thm:sp2-covering}]
In \cite[Proposition 3.11]{haagerupdelaat2}, it was proved that for all functions $\varphi$ in $M_0A(\widetilde{G}) \cap \mathcal{C}$ and $t \in \mathbb{R}$, the limit
$$
c_{\varphi}(t)=\lim_{s \to \infty} \dot{\varphi}(2s,s,t)
$$
exists. Moreover, there exist constants $C_1,C_2 > 0$ such that for all $\beta \geq \gamma \geq 0$, we have
\[
  |\dot{\varphi}(\beta,\gamma,t)-c_{\varphi}(t)| \leq C_1 e^{-C_2\sqrt{\beta^2+\gamma^2}}\|\varphi\|_{M_0A(\widetilde{G})}.
\]
It was also proved, in \cite[Proposition 3.33]{haagerupdelaat2}, that $c_{\varphi}$ is a constant function.

Now, let $\varphi \in M_0A(\widetilde{G}) \cap \mathcal{C}$, and suppose that $c_{\varphi} = 0$. We prove that $m(\varphi)=0$. Let $\epsilon > 0$ be given. Since $\varphi$ is weakly almost periodic, there exist $g_1,\ldots,g_n \in\widetilde G$ and $c_1,\ldots,c_n \geq 0$ with $\sum_{i=1}^n c_i=1$ such that $\|\sum_{i=1}^n c_iL_{g_i}\varphi - m(\varphi)\|_{\infty} < \epsilon$. Since $c_{\varphi} = 0$, there exists an $R > 0$ such that for all $\beta\geq\gamma\geq0$ with $\beta^2 + \gamma^2 \geq R$ and all $t \in \mathbb{R}$, we have $|\varphi(\widetilde{v}_t\widetilde{D}(\beta,\gamma))| < \epsilon$. Put
$$
V=\{\widetilde{h}_1\widetilde{v}_t\widetilde{D}(\beta,\gamma)\widetilde{h}_2 \mid t \in \mathbb{R},\, \widetilde{h}_1,\widetilde{h}_2 \in \widetilde{H},\, \beta\geq\gamma\geq 0,\, \beta^2 + \gamma^2 < R\}.
$$
The set $V$ is $\widetilde{H}$-bi-invariant. By \cite[Lemma 3.7]{haagerupdelaat2}, it follows that $|\varphi(g)| < \epsilon$ for $g \in \widetilde{G} \setminus V$. Consider the quotient map $\sigma\colon\widetilde G \to G$. Then
$$
\sigma(V) \subset \{k_1D(\beta,\gamma)k_2 \mid \beta\geq\gamma\geq 0,\, \beta^2+\gamma^2 \leq R,\, k_1,k_2 \in K\},
$$
which is compact. Hence, finitely many translates of $\sigma(V)$ cannot cover $G$, so finitely many translates of $V$ cannot cover $\widetilde{G}$. Now, let $g_1,\ldots,g_n \in \widetilde{G}$ be as above, and choose $g \in \widetilde{G} \setminus (\cup_{i=1}^n g_iV)$. Then $|\varphi(g_i^{-1}g)| < \epsilon$ for $i=1,\ldots,n$. It follows that $|\sum_{i=1}^n c_i\varphi(g_i^{-1}g)| < \epsilon$, where $c_1,\ldots,c_n \geq 0$ are as above. Hence, we conclude that $|m(\varphi)| < 2\epsilon$.

We have now shown that, when restricted to $M_0A(\widetilde G)\cap\mathcal C$, the kernel of $m$ consists precisely of those $\phi\in M_0A(\widetilde G)\cap\mathcal C$ for which $c_\phi = 0$. By \cite[Lemma~3.12]{haagerupdelaat2}, this set is weak$^*$ closed, and hence the restriction of $m$ to $M_0A(\widetilde G)\cap\mathcal C$ is weak$^*$ continuous.

It still remains to prove that $m\colon M_0A(\widetilde G)\to\mathbb C$ is weak$^*$ continuous. As in \cite[Lemma~3.10]{haagerupdelaat2}, for $\phi\in M_0A(\widetilde G)$, we define
$$
\phi^{\mathcal C}(g) = \frac{1}{\pi} \int_0^\pi\!\! \int_{\widetilde H} \int_{\widetilde H} \phi(\widetilde{h}_1\widetilde v_t g \widetilde v_t^{-1} \widetilde{h}_2) \ d\widetilde{h}_1 d\widetilde{h}_2 dt,
\quad g\in \widetilde G,
$$
where $d\widetilde h_1$ and $d\widetilde h_2$ both denote the normalized Haar measure on the compact group $\widetilde H$. Then $\phi^{\mathcal C} \in M_0A(\widetilde G)\cap\mathcal C$ and $\phi\mapsto\phi^{\mathcal C}$ is weak$^*$--weak$^*$ continuous. It is not difficult to see that $m(\phi) = m(\phi^{\mathcal C})$ by adapting the proof of Lemma \ref{lem:average}. It now follows that $m$ is weak$^*$ continuous. Hence, $\widetilde{G}$ has property (T$^*$).
\end{proof}
The above three theorems imply Theorem \ref{thm:sl3sp2sp2cov}.

\section{The Approximation Property for connected Lie groups} \label{sec:connectedliegroups}
In this section, we prove Theorem \ref{thm:apconnectedliegroups}, which provides the complete characterization of connected Lie groups with the AP. For details about the Levi decomposition, we refer to Section \ref{subsec:levidecomposition}.

\begin{repthm}{thm:apconnectedliegroups}
Let $G$ be a connected Lie group, let $G=RS$ denote its Levi decomposition, and suppose that $S$ is locally isomorphic to the product $S_1 \times \cdots \times S_n$ of connected simple factors. Then the following are equivalent:
\begin{enumerate}[(i)]
 \item the group $G$ has the AP,
 \item the group $S$ has the AP,
 \item the groups $S_i$, with $i=1,\ldots,n$, have the AP,
 \item the real rank of the groups $S_i$, with $i=1,\ldots,n$, is at most $1$.
\end{enumerate}
\end{repthm}
\begin{proof}
The equivalence of (iii) and (iv) is \cite[Theorem~5.1]{haagerupdelaat2}.

Suppose (iv) holds. We show that (i) and (ii) hold. Recall that the AP is preserved under group extenstions \cite[Theorem~1.15]{haagerupkraus}. Since the group $R$ is solvable, it has the AP. Thus, to prove (i), it is enough to prove that $G/R$ has the AP. The Lie algebra of $G/R$ is the Lie algebra of $S$, so to prove (i) and (ii), it is sufficient to prove that any connected Lie group $S'$ locally isomorphic to $S_1\times\cdots\times S_n$ has the AP. Again because of \cite[Theorem~1.15]{haagerupkraus}, we may assume that the center of $S'$ is trivial. If $Z_i$ denotes the center of $S_i$, then $S' \cong (S_1/Z_1)\times\cdots\times(S_n/Z_n)$. Each group $S_i/Z_i$ has real rank at most $1$ and hence has the AP \cite[Theorem~5.1]{haagerupdelaat2}. Hence $S'$ has the AP.

Suppose (iv) does not hold. We show that (i) and (ii) do not hold. For $i=1,\ldots,n$, let $\widetilde S_i$ denote the universal covering group of $S_i$ and put $\widetilde S = \widetilde S_1\times\cdots\times\widetilde S_n$. Then $\widetilde S$ is the universal covering group of $S$. Let $\pi\colon\widetilde S\to S$ denote the covering homomorphism.

Since the real rank of some $\widetilde S_i$ is at least $2$, there is a closed subgroup $\widetilde H\subset\widetilde S_i$ that is locally isomorphic to $\SL(3,\mathbb R)$ or $\Sp(2,\mathbb R)$ (see e.g.~\cite{boreltits} or \cite{dorofaeff}). Since $\widetilde H$ is a quotient of either $\widetilde\SL(3,\mathbb R)$ or $\widetilde\Sp(2,\mathbb R)$ (the universal covering groups of $\SL(3,\mathbb R)$ and $\Sp(2,\mathbb R)$, respectively), it follows from Corollary~\ref{cor:quotients}, Remark~\ref{rmk:cover-SL3}, and Theorem~\ref{thm:sp2-covering} that $\widetilde H$ has property (T$^*$). If we let $H = \pi(\widetilde H)$, then $H$ is a Lie subgroup of $S$ locally isomorphic to $\widetilde H$. Such a subgroup of $S$ must be closed \cite[Corollary~1]{dorofaeff}. By Corollary~\ref{cor:quotients}, $H$ has property (T$^*$). Since the group $H$ is locally isomorphic to $\widetilde H$, it is not compact. It follows from Proposition~\ref{prp:obstruction} that $H$ and hence $S$ does not have the AP.

If $\overline H$ denotes the closure of $H$ in $G$, then $\overline H$ has property (T$^*$) by Proposition~\ref{prp:denseimage}. Also, $\overline H$ is not compact, since this would imply that $H$ was closed in $\overline H$ (see \cite[p.~615]{mostow}) and hence compact. It follows from Proposition~\ref{prp:obstruction} that $\overline H$ and hence $G$ does not have the AP.

The proof is now complete.
\end{proof}

We remark that the proof of Theorem~\ref{thm:apconnectedliegroups} shows moreover that a connected Lie group without the AP contains a non-compact closed subgroup with property (T$^*$).

\section{Property \texorpdfstring{(T$^{\ast}$)}{(T*)} for connected simple higher rank Lie groups} \label{sec:tstarhigherrankliegroups}
Using the result of Veech (see Theorem \ref{thm:veech}) that was used in Section~\ref{sec:sl3sp2sp2cov}, we are able to establish Property (T$^*$) for connected simple Lie groups with real rank at least $2$ and finite center.
\begin{repthm}{thm:tstarhigherrank}
Let $G$ be a connected simple Lie group with real rank at least $2$ and finite center. Then $G$ has property (T$^*$).
\end{repthm}
\begin{proof}
We use again that $G$ contains a closed subgroup $H$ that is locally isomorphic to $\SL(3,\mathbb R)$ or $\Sp(2,\mathbb R)$. Then $H$ has property (T$^*$). Let $m_H$ denote the unique left invariant mean on $W(H)$, which is weak$^*$ continuous on $M_0A(H)$.

Fix a function $g\in C_c(G)_+$ such that $\int_G g(x) dx = 1$ and $g^*=g$. Define $T\colon W(G)\to W(H)$ by
$$
T\phi = (g*\phi)|_H, \qquad \phi\in W(G).
$$
It is clear that $T(C_0(G))\subset C_0(H)$, and hence $(m_H\circ T)(C_0(G)) = 0$. It now follows from Veech's result that $m_G = m_H\circ T$. To finish the proof, we argue that the restriction of $T$ to $M_0A(G)$ is a weak$^*$--weak$^*$ continuous map to $M_0A(H)$. But this is Lemma~\ref{lem:convolution} (take $\pi$ to be the inclusion map).
\end{proof}
We expect that the above theorem holds without the finite center condition.

\end{document}